\documentclass[12pt]{amsart}

\newif\ifdebug                                                      %
\debugfalse

\newcommand{\printname}[1]
   {\smash{\makebox[0pt]{\hspace{-1.0in}\raisebox{8pt}{\tiny #1}}}}

\newcommand{\labell}[1] {\label{#1}{\ifdebug{\printname{#1}}\fi}}

\usepackage{amssymb}
\usepackage{amsmath}
\usepackage{amscd}
\usepackage{wasysym}

\usepackage{epic}
\usepackage{eepic}
\usepackage[all]{xy}

\usepackage[utf8]{inputenc}

\def \v {{\lambda;\delta_1,\ldots,\delta_k}}
\def \R  {{\mathbb R}}
\def \Z  {{\mathbb Z}}
\def \C  {{\mathbb C}}

\def \CP {{{\mathbb C}{\mathbb P}}}

\def \CP {{\mathbb C}{\mathbb P}}
\def \t  {{\mathfrak t}}

\def \calJ {{\mathcal J}}

\def    \half   {{\frac{1}{2}}}

\def \Symp {{\text{Symp}}}

\DeclareMathOperator \PSL {PSL}

\DeclareMathOperator \GL {GL}

\DeclareMathOperator \id {id}
\DeclareMathOperator \AGL {AGL}
\DeclareMathOperator \GW {GW}

\numberwithin{equation}{section}
\numberwithin{figure}{section}

\makeatletter
\let\c@equation\c@figure
\makeatother

\swapnumbers

\newtheorem{Lemma}[equation]{Lemma}
\newtheorem{Theorem}[equation]{Theorem}
\newtheorem*{thm*}{Theorem}

\newtheorem{Corollary}[equation]{Corollary}
\newtheorem{Claim}[equation]{Claim}
\newtheorem*{Lemma*}{Lemma}
\newtheorem*{Corollary*}{Corollary}

\theoremstyle{remark}

\newtheorem*{Remark*}{Remark}

\theoremstyle{definition}

\begin{document}

\title[Counting toric actions on symplectic four-manifolds]
{Counting toric actions on symplectic four-manifolds}

\date{\today}


\author{Yael Karshon}
\address{Department of Mathematics, University of Toronto,
Toronto, Ontario, Canada, M5S 3G3.}
\email{karshon@math.toronto.edu}

\author{Liat Kessler}
\address{Department of Mathematics, Physics, and Computer Science, 
University of Haifa, at Oranim, Tivon 36006, Israel}
\email{lkessler@math.haifa.ac.il}

\author{Martin Pinsonnault}
\address{Department of Mathematics, Middlesex College
The University of Western Ontario London, Ontario N6A 5B7 Canada}
\email{mpinson@uwo.ca}

\thanks{The first and third authors are partially supported by 
the Natural Science and Engineering Research Council of Canada.
The second author was partially supported 
by the Center for Absorption in Science, 
Ministry of Immigrant Absorption, State of Israel.}

\begin{abstract}
Given a symplectic manifold, we ask in how many different ways can a torus
act on it.
Classification theorems in equivariant symplectic geometry can sometimes
tell that two Hamiltonian torus actions are inequivalent, but often
they do not tell whether the underlying symplectic manifolds are
(non-equivariantly) symplectomorphic.
For two dimensional torus actions on closed symplectic four-manifolds,
we reduce the counting question to combinatorics, by expressing the manifold
as a symplectic blowup in a way that is compatible with all the torus
actions simultaneously.  
For this we use the theory of pseudoholomorphic curves.

\smallskip

Nous nous int\'eressons aux diff\'erentes actions d'un tore sur une 
vari\'et\'e symplectique donn\'ee.
En g\'eom\'etrie symplectique \'equivariante, les th\'eor\`emes de
classification permettent parfois de distinguer des actions hamiltoniennes
de tores g\'e\-o\-m\'e\-tri\-que\-ment in\'equivalentes. Par contre,
ces th\'eor\`emes ne permettent habituellement pas de d\'eterminer si
les vari\'et\'es symplectiques sous-ja\c{c}entes sont symplectomorphes.
Dans le cas des vari\'et\'es symplectiques de dimension $4$, nous 
r\'eduisons le probl\`eme d'\'enum\'eration des actions toriques 
in\'equivalentes \`a un probl\`eme combinatoire en exprimant la 
vari\'et\'e consid\'er\'ee comme un \'eclatement symplectique qui est 
compatible simultan\'ement avec toutes les actions toriques. Ce r\'esultat 
est obtenu en employant des techniques pseudo-holomorphes.

\end{abstract}

\maketitle

\section{Introduction}
\labell{sec:intro}

An \textbf{action} of a torus $T \cong (S^1)^k$ 
on a symplectic manifold $(M,\omega)$
is a group homomorphism $\rho \colon T \to \Symp(M)$ 
that is smooth in the diffeological sense:  
the map $(a,m) \mapsto \rho(a)(m) =: a \cdot m$ from $T \times M$ to $M$ 
is smooth.
A \textbf{momentum map} for such an action is a map
$\Phi \colon M \to \t^* \cong \R^k$
such that $ d \Phi_j = - \iota(\xi_j) \omega$ for all $j=1,\ldots,k$, 
where $\xi_1,\ldots,\xi_k$ are the vector fields
that generate the torus action.
When considering torus actions on symplectic manifolds, 
we will always assume 
that \emph{the action is faithful}, i.e., $\rho$ is one-to-one; 
that \emph{the manifold $M$ is connected};
and that \emph{the action is Hamiltonian}, 
i.e., that a momentum map exists.
%
If $\dim T = \half \dim M$, the Hamiltonian $T$-action 
is called \textbf{toric}.
If, additionally, $M$ is compact, 
then the image of the momentum map is a unimodular (``Delzant") polytope,
and this polytope determines the triple $(M,\omega,\Phi)$
up to an equivariant symplectomorphism that respects the momentum map;
this is Delzant's theorem \cite{delzant}.
It follows that two toric actions differ by conjugation in $\Symp(M)$ 
and reparametrization of $T$ if and only if their momentum images 
are $\AGL(n,\Z)$-congruent where $n=\dim T = \frac{1}{2}\dim M$, 
that is, they differ by a transformation of the form $x \mapsto Ax+b$
where $A$ is in $\GL(n,\Z)$ and $b \in \R^n$;
we consider such actions as \textbf{equivalent}.

Up to equivalence, a compact symplectic four-manifold has
only finitely many toric actions. We prove this fact
in \cite{kkp}; the proof uses ``soft'' equivariant, algebraic,
and combinatorial methods.
This finiteness result remains true in higher dimensions;
this is proved by Borisov and McDuff in \cite[Proposition 3.1]{mb}.
In dimension four, the number of inequivalent Hamiltonian circle actions
that do not extend to 2-torus actions is also finite;
this is proved by Pinsonnault \cite{pinso}.

If the second Betti number of a compact symplectic four-manifold
is one or two, a ``soft" argument gives the exact number
of inequivalent toric actions; see \cite{YK:max_tori}.
But in more general cases the proofs in \cite{kkp,pinso,mb}
that the number of actions is finite
do not give us the actual number of nonequivalent toric actions; 
in general they do not even determine whether this number is nonzero,
that is, whether toric actions exist.

Whether toric or circle actions exist was determined for symplectic
manifolds that are obtained from a $\CP^2$ of size $\lambda$
by $k$ symplectic blowups of equal size $\epsilon$, for many values of
$\lambda$ and $\epsilon$,  in \cite{kk}, \cite{ke}, and \cite{pinso}: 
if $k \geq 4$ then the manifold does not admit a $2$-torus action, and
if $(k-1)\epsilon \geq \lambda$ then the manifold does not admit a circle
action. (The ``size" of a blowup is $\frac{1}{2\pi}$ times the symplectic
area of the exceptional divisor; the ``size" of $\CP^2$
is $\frac{1}{2\pi}$ times the symplectic area of $\CP^1 \subset \CP^2$.)
The challenge is to show that there are no ``exotic actions":
every action can be made ``consistent with the blow-ups"
so that it comes from an action on $\CP^2$.
The proofs combine ``soft'' and ``hard'' (holomorphic) techniques.

In our more recent work
we show that there are no ``exotic actions" also for arbitrary
blowups of $\CP^2$.  
In fact, 
every compact connected symplectic four-manifold that admits a toric action
and whose second Betti number is $\geq 3$
is (non-equivariantly) symplectomorphic 
to a symplectic manifold that is obtained from $\CP^2$ 
with a multiple of the Fubini-Study form by a sequence of symplectic blowups.
(This follows from Delzant's theorem \cite{delzant}
and observations of \cite[Section 2.5]{fulton} and \cite[Lemma 3]{YK:max_tori};
see \cite[Corollary 2.17]{kkp}.)
Thus, our work yields an algorithm to count the exact number
of inequivalent toric actions on an \emph{arbitrary} 
compact symplectic four-manifold.
In this paper we report on this work, deferring details
to a longer and more leisurely exposition~\cite{kkp2}.

We denote by 
$$ (M_k,\omega_{\lambda;\delta_1,\ldots,\delta_k}) $$
a symplectic manifold that is obtained from
a $\CP^2$ of size $\lambda$ by blowups of sizes $\delta_1,\ldots,\delta_k$.
If such a manifold exists, 
then it is unique up to symplectomorphism, by McDuff's work \cite{isotopy}.

To compare different blowups, it is convenient to fix the underlying
manifold $M_k$.
Once and for all, we fix a sequence $p_1,p_2,p_3,\ldots$ of distinct points
on the complex projective plane $\CP^2$, and we denote by $M_k$
the manifold that is obtained from $\CP^2$ by complex blowups
at $p_1,\ldots,p_k$.  We have a decomposition
$$ H_2(M_k) = \Z L \oplus \Z E_1 \oplus \ldots \oplus \Z E_k $$
where $L$ is the image of the homology class of a line $\CP^1$
in $\CP^2$ under the inclusion map $H_2(\CP^2) \to H_2(M_k)$
and where $E_1,\ldots,E_k$ are the homology classes of the exceptional
divisors. 

Throughout this paper, homology is taken with integer coefficients 
and cohomology is taken with real coefficients.

\bigskip

A \textbf{blowup form} on $M_k$ is a symplectic form for which 
there exist pairwise disjoint 
embedded symplectic spheres in the classes $L,E_1,\ldots,E_k$.

\bigskip

Fix a non-negative integer $k$.
Let $\left< \cdot , \cdot \right>$ denote the pairing
between cohomology and homology on $M_k$.
A vector $(\v)$ in $\R^{1+k}$
\textbf{encodes} a cohomology class $\Omega \in H^2(M_k;\R)$
if $\frac{1}{2\pi} \left< \Omega , L \right> = \lambda$
and $\frac{1}{2\pi} \left< \Omega , E_j \right> = \delta_j$
for $j= 1, \ldots, k$.
\bigskip 

Thus, $\omega_{\v}$ can be taken to be a blowup form on $M_k$
whose cohomology class is encoded by the vector $(\v)$.
It is unique up to a diffeomorphism that acts trivially on the homology, 
as follows from results of Gromov \cite[2.4.A', 2.4.A1']{gromovcurves} 
and McDuff \cite{rational-ruled,isotopy}.

\bigskip
Let $k \geq 3$, and let $\lambda, \delta_1, \ldots, \delta_k$
be real numbers.
The vector $(\lambda ; \delta_1 , \ldots , \delta_k)$
is \textbf{reduced} if
\begin{equation} \labell{conditions-1}
 \delta_1 \geq \ldots \geq \delta_k \quad \text{ and } \quad
   \delta_1 + \delta_2 + \delta_3 \leq \lambda. \nonumber
\end{equation}

In \cite[Theorem 1.4]{2paper} we show the following result:
\begin{Theorem}
Let $k \geq 3$.  Given a blowup form 
$\omega_{\lambda';\delta_1',\ldots,\delta_k'}$ on $M_k$,
there exists a unique reduced vector $(\lambda;\delta_1,\ldots,\delta_k)$
such that
$ (M_k,\omega_{\lambda';\delta_1',\ldots,\delta_k'})
 \cong (M_k,\omega_{\lambda;\delta_1,\ldots,\delta_k}). $
\end{Theorem}

Here is our main result about the non-existence of ``exotic actions":

\begin{Theorem} \labell{theorem-2}
Let $k$ be an integer $\geq 3$. Let
$\omega_{\lambda;\delta_1,\ldots,\delta_k}$ be a blowup form on $M_k$
whose cohomology class is encoded by a reduced vector
 $(\lambda;\delta_1,\ldots,\delta_k)$.
Let
\begin{align} \labell{nota3}
\delta & = \lambda - \delta_1 - \delta_2 , \nonumber \\
 a & = \lambda - \delta_2 , \\
 b & = \lambda - \delta_1. \nonumber
\end{align}
Then every toric action
on $(M_k,\omega_{\lambda;\delta_1,\ldots,\delta_k})$
is isomorphic to one that is obtained from a toric action
on $(S^2 \times S^2 , a \omega_{S^2} \oplus b \omega_{S^2})$
by a sequence of equivariant symplectic blowups
of sizes $\delta,\delta_3,\ldots,\delta_k$.
\end{Theorem}

Moreover, Theorem \ref{theorem-2} holds if we replace ``toric action"
with ``action of a (not necessarily connected) 
compact Lie group $G$ that preserves $\omega$ and induces 
the identity morphism on $H_2(M)$".

The (equivalence classes of) toric actions on 
$(S^2 \times S^2, a \omega_{S^2} \oplus b \omega_{S^2})$ 
are enumerated by the set of integers $\ell$ 
such that $0 \leq \ell < \frac{a}{b}$; 
their momentum map images are the Hirzebruch trapezoids $H_{a,b,2 \ell}$, 
with height $b$, average width $a$, left edge vertical, 
and right edge of slope $-{1 \over 2\ell}$ or vertical if $\ell=0$
\cite[Theorem~2]{YK:max_tori}; see Figure~\ref{fig:hirz}.

\begin{figure}
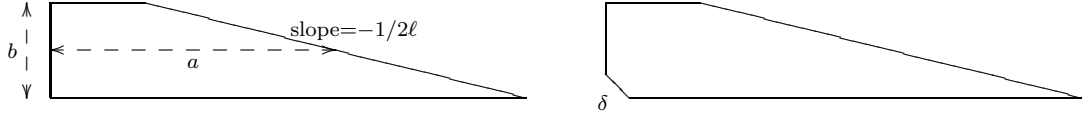

\begin{center}
$$
\xy
0;/r.15pc/: 
(0,0)*{}; (100,0)*{} **\dir{-};
(0,0)*{}; (0,20)*{} **\dir{-};
(0,20)*{}; (20,20)*{} **\dir{-};
{\ar^{\text{slope=$-1/2\ell$}}@{-} (20,20)*{}; (100,0)*{}};
{\ar^{b}@{<-->} (-5,0)*{}; (-5,20)*{} };
{\ar_{a}@{<-->} (0,10)*{}; (60,10)*{} };
\endxy 
\qquad
\xy
0;/r.15pc/: 
(5,0)*{}; (100,0)*{} **\dir{-};
(0,5)*{}; (0,20)*{} **\dir{-};
(0,20)*{}; (20,20)*{} **\dir{-};
{\ar@{-} (20,20)*{}; (100,0)*{}};
{\ar_{\delta}@{-} (0,5)*{}; (5,0)*{} };
\endxy 
$$
\end{center}
\caption{A ``Hirzebruch trapezoid" and the effect of an equivariant blowup}
\labell{fig:hirz}
\end{figure}

The effect of an equivariant blowup of size $\epsilon$ on the momentum
mapping image $\Delta$ amounts to ``chopping off a corner of size $\epsilon$",
see Figure~\ref{fig:hirz}. This can be done at a vertex of $\Delta$ if and
only if all the edges that emanate from that vertex have size $> \epsilon$.
Here, the ``size" of a vector of rational slope is the positive number
such that the vector is equal to that number
times a primitive lattice element.

Given a unimodular (``Delzant") polygon,
choose a cyclic ordering of its edges,
let $u_1,\ldots,u_N$ be the primitive inner normals to the edges,
and let $a_1,\ldots,a_N$ be the sizes of the edges.
There exist integers $k_j$ such that $u_{j+1} + u_{j-1} = -k_j u_j$.
(If the polygon is the momentum image for a toric action
on a symplectic manifold $M$, the preimages of the edges 
are embedded symplectic two-spheres in $M$, 
the $k_j$ are their self intersections,
and the $a_j$ are $\frac{1}{2\pi}$ times their symplectic areas.)
Two polygons are $\AGL(2,\Z)$-congruent
if and only if their corresponding vectors $\{ (k_j,a_j) \}$
differ by a cyclic or anti-cyclic permutation.

Therefore, by Theorem \ref{theorem-2} and Delzant's theorem, 
we obtain the list of (equivalence classes of) toric actions on
$(M_k,\omega_{\lambda;\delta_1,\ldots,\delta_k})$
from the following combinatorial algorithm.
Recursively produce all the polygons that are obtained by a sequence of corner
choppings of sizes $\delta,\delta_3,\ldots,\delta_k$ starting from a
Hirzebruch trapezoid $H_{a,b,2\ell}$, where $a,b,\delta$ are 
as in~\eqref{nota3} and $0 \leq \ell < \frac{a}{b}$.
Use the associate vectors to eliminate repetition up to $\AGL(2,\Z)$.

Here is a sample of some consequences of this algorithm.  
Let $k$ be an integer $\geq 3$.
Let $\omega$ be a blowup form on ${M_k}$
whose cohomology class is encoded by a reduced vector $(\v)$.
\begin{enumerate}
\item
Suppose that
$$ \lambda - \delta_1-\delta_2=\delta_{3} = \delta_4=\delta_5 = \delta_{6},$$
or that  there exists $i \geq 1$ such that
$$ \delta_{i} = \delta_{i+1} = \ldots = \delta_{i+(i+2)}. $$
Then there is no toric action on $(M_k,\omega)$.

\item
The number of toric actions on $(M_k,\omega)$ is at most
\begin{equation} \labell{upper bound}
\left( \left\lceil \frac{\lambda-\delta_2}{\lambda-\delta_1} \right\rceil
        + \left\lceil \frac{\delta_1-\delta_2}{\lambda-\delta_1} \right \rceil
\right)
   \cdot \frac{(k+2)!}{24}.
\end{equation}
This upper bound is achieved if and only if the following conditions (i)--(iv)
hold.
\begin{enumerate}
\item[(i)]
$ \left\lceil \frac{\lambda-\delta_2}{\lambda-\delta_1} \right\rceil
        \cdot (\lambda-\delta_1)  <  \lambda $,
\item[(ii)]
$ \delta_{j+1} + \ldots + \delta_k  <  \delta_j
 \quad \text{ for all } 3 \leq j < k $,
\item[(iii)]
$ \delta_3 + \ldots + \delta_k  <  \lambda - \delta_1 - \delta_2 $,
 \quad \text{ and }
\item[(iv)]
$ \delta_3 + \ldots + \delta_k
        <  \lambda - \left\lceil \frac{\lambda-\delta_2}{\lambda-\delta_1}
                   \right \rceil
                   \cdot (\lambda-\delta_1)$.
\end{enumerate}
For example, this upper bound is  attained 
for $\lambda = 1$, \ $\delta_1 = \delta_2 = \frac{1}{x}$,
and $\delta_i = \frac{1}{x^{i-1}}$ \ for $3 \leq i \leq k$,
where $x \geq \frac{3+\sqrt{5}}{2}$.
\end{enumerate}

\section{Sketch of proof of Theorem \ref{theorem-2}}

We use holomorphic tools in almost complex four-manifolds. 
An \textbf{almost complex structure} on a manifold $M$ is an automorphism $J
\colon TM \to TM$ such that $J^2=-\id$. An almost complex structure $J$
on $M$ is \textbf{tamed} by a symplectic form $\omega$ if $\omega(u,Ju)$
is positive for all nonzero $u \in TM$; let $\calJ_\tau(M,\omega)$
denote the set of almost complex structures that are tamed by $\omega$.
A \textbf{$J$-holomorphic sphere} is a $J$-holomorphic map from
$\CP^1$ to $M$; it is \textbf{simple} if it cannot be factored through a
branched covering of the domain.
The (genus zero) \textbf{Gromov-Witten invariant} 
of a compact four dimensional manifold $(M,\omega)$
associates to every homology class $A \in H_2(M)$ an integer $\GW(A)$.
Let $\kappa(A) = ( A\cdot A + c_1(TM)(A) )/2$.
For generic 
$(J,p_1,\ldots,p_{\kappa(A)}) \in \calJ_\tau(M,\omega) \times M^{\kappa(A)}$,
the integer $GW(A)$ is equal to the number of simple $J$-holomorphic 
spheres
$\CP^1 \to M$ in the class $A$ that pass
through the points $p_1,\ldots,p_{\kappa(A)}$,
modulo reparametrization by an element of $\PSL(2,\C)$,
and counted with appropriate signs.

\bigskip

Let $k \geq 3$.  Let $\omega$ be a blowup form on ${M_k}$
whose cohomology class is encoded by a vector $v=(\v)$ that is reduced.

\begin{Lemma} \labell{deltak is minimal}
Let $A$ be a class in $H_2(M_k)$.
Suppose that $c_1(TM_k)(A) \geq 1$,
and suppose that $A$ is represented by a nontrivial $J$-holomorphic sphere 
for some almost complex structure $J$ that is tamed by some blowup form 
on $M_k$.  Then
\begin{equation} \labell{geq deltak}
 \frac{1}{2\pi} \left< [\omega] , A \right> \geq \delta_k .
\end{equation}
\end{Lemma}

\begin{proof}
See \cite[Lemma 3.6]{2paper} or~\cite{kkp2}.
\end{proof}

\begin{Theorem} \labell{indecomposable1}
Let $A$ be a class in $H_2(M_k)$ such that $c_1(TM_k)(A) > 0$
and $\GW(A) \neq 0$.
Suppose that
$$ \frac{1}{2\pi} \left< [\omega] , A \right> = \delta_k .$$
Then for every almost complex structure $J$ that is $\omega$-tamed
there exists a $J$-holomorphic sphere in the class $A$.
\end{Theorem}

\begin{proof}
Let $J$ be an almost complex structure that is $\omega$-tamed.
Because $\GW(A) \neq 0$,
and by Gromov's compactness theorem \cite[1.5.B]{gromovcurves}, 
there exist classes
$A_1, \ldots, A_\ell \in H_2(M_k)$ such that
$$ A = A_1 + \ldots + A_\ell $$
and such that each $A_j$ can be represented
by a nonconstant $J$-holomorphic sphere.
We claim that $\ell = 1$.
Because $c_1(TM_k)(A) > 0$,
there exists at least one summand $A_j$ such that $c_1(TM_k)(A_j) > 0$.
Fix such an $A_j$.  Then $c_1(TM_k)(A_j) \geq 1$.
By Lemma~\ref{deltak is minimal},
the $j$th summand in the sum
$\sum_{i=1}^\ell \frac{1}{2\pi} \left< [\omega] , A_i \right>$
is $\geq \delta_k$.  But the entire sum is equal to $\delta_k$
and all the summands are strictly positive.
So $\ell = 1$, as required.
\end{proof}

A homology class $E \in H_{2}(M_k)$ is \textbf{exceptional} if it is
represented by an embedded $\omega$-symplectic sphere with self
intersection $-1$. Note that $E_1,\ldots,E_k$ are exceptional.

\begin{Corollary}  \labell{Ek embeds}
For every almost complex structure $J$ that is $\omega$-tamed
there exists an embedded $J$-holomorphic sphere in the class $E_k$.

Moreover,
let $E$ be \emph{any} exceptional class in $H_2(M_k)$ such that
$ \frac{1}{2\pi} \left< [\omega] , E \right> = \delta_k $.
Then for every almost complex structure $J$ that is $\omega$-tamed
there exists an \emph{embedded} $J$-holomorphic sphere in the class $E$.
\end{Corollary}

\begin{proof}
For an exceptional class $E \in H_{2}(M_k)$, we have $c_1(TM_k)(E) = 1$
and, by McDuff's ``$C_1$ lemma'' \cite[Lemma 3.1]{rational-ruled} and
Gromov's compactness theorem \cite[1.5.B]{gromovcurves}, the invariant
$\GW(E)$ is nonzero. Therefore we can apply Theorem \ref{indecomposable1}
to $E$. By the adjunction formula \cite[Corollary 1.7]{nsmall},
a $J$-holomorphic sphere in $E$ is embedded.
\end{proof}

\begin{Lemma} \labell{lemg}
Let $E \in H_2(M_k)$
such that $E \cdot E=-1$.
Let a compact Lie group $G$ act on $M_k$.
Suppose that the $G$ action preserves $\omega$
and induces the identity morphism on  $H_2(M_k)$.
Let $J_G$ be a $G$-invariant $\omega$-tamed almost complex structure on $M_k$.
Let $C$ be an embedded $J_G$-holomorphic sphere in the class~$E$.
Then $C$ is a $G$-invariant $\omega$-symplectic embedded sphere.
\end{Lemma}

\begin{proof}
Because $J_G$ is $\omega$-tamed
and $C$ is an embedded $J_G$-holomorphic sphere,
$C$ is an embedded $\omega$-symplectic sphere.

Let $a \in G$.  Because $G$ acts trivially on the homology,
$[aC] = [C] = E$.
By positivity of intersections of J-holomorphic spheres in an almost
complex four-manifold \cite[Proposition 2.4.4]{nsmall}, 
and since $E \cdot E = -1$,
the spheres $aC$ and $C$ must coincide.
Thus, $C$ is $G$-invariant.
\end{proof}

\begin{Corollary} \labell{corg}
Let a compact Lie group $G$ act on $M_k$, preserve $\omega$,
and act trivially on $H_2(M_k)$.
Then there exists a $G$-invariant $\omega$-symplectic embedded sphere $C$
in the class~$E_k$. 
\end{Corollary}

Equivariantly blowing down along the sphere $C$ yields a $G$-action 
on $(M_{k-1},\omega_{\lambda;\delta_1,\ldots,\delta_{k-1}})$.
By repeated $k-2$ applications
of Corollary~\ref{corg}, we reduce Theorem \ref{theorem-2} 
to the following claim.
\begin{Claim}
Let $\omega$ be a blowup form on $M_2$,
whose cohomology class is encoded by a vector
$(\lambda;\delta_1,\delta_2)$ with $\delta_1 \geq \delta_2$.
Let a compact Lie group $G$ act on $M_2$, preserve $\omega$, 
and act trivially on $H_{2}(M_2)$.

Then there exists an embedded
$G$-invariant $\omega$-symplectic sphere $D$
in the class $L - E_1 - E_2$, 
and blowing down along $D$ yields a $G$-action on 
$(S^2 \times S^2 , a \omega_{S^2} \oplus b \omega_{S^2})$
where
$a  = \lambda - \delta_2$ and 
 $b = \lambda - \delta_1$.
\end{Claim}

The claim is proved by combinatorial tools in case the action is toric 
and by holomorphic tools in the general case.
See~\cite[Lemma~2.2]{pinso-2} and~\cite{kkp2}.

\subsection*{Acknowledgement}
The authors are grateful for useful discussions with
Paul Biran, Tian-Jun Li, Dusa McDuff, and Dietmar Salamon.

\end{document}